\UseRawInputEncoding
\documentclass[12pt, reqno]{amsart}
\usepackage[margin=1in]{geometry}
\usepackage{amssymb,latexsym,amsmath,amscd,amsfonts}
\usepackage{latexsym}
\usepackage[mathscr]{eucal}
\usepackage{bm}
\usepackage{mathptmx}
\usepackage{amssymb}
\usepackage{amsthm}
\usepackage{dcolumn}
\usepackage[all]{xy}
\usepackage{enumitem}
\usepackage[utf8]{inputenc}

\def\textmatrix#1&#2\\#3&#4\\{\bigl({#1 \atop #3}\ {#2 \atop #4}\bigr)}
\def\dispmatrix#1&#2\\#3&#4\\{\left({#1 \atop #3}\ {#2 \atop #4}\right)}
\newcommand{\beg}{\begin{equation}}
	\newcommand{\eeg}{\end{equation}}
\newcommand{\ben}{\begin{eqnarray*}}
	\newcommand{\een}{\end{eqnarray*}}

\newtheorem{thm}{Theorem}[section]

\numberwithin{equation}{section} \theoremstyle{definition}

\newcommand{\Dbar}{\overline{\mathbb{D}}}
\newcommand{\C}{\mathbb{C}}

\newcommand{\D}{\mathbb{D}}
\newcommand{\T}{\mathbb{T}}

\newcommand{\Gg}{\mathbb{G}_n}

\newcommand{\E}{\mathbb E}
\newcommand{\Ebar}{\overline{\mathbb E}}
\newcommand{\Pe}{\mathbb P}
\newcommand{\Pbar}{\overline{\mathbb P}}
\newcommand{\G}{\mathbb{G}_2}

\newcommand{\w}{\omega}

\def\textmatrix#1&#2\\#3&#4\\{\bigl({#1 \atop #3}\ {#2 \atop #4}\bigr)}
\def\dispmatrix#1&#2\\#3&#4\\{\left({#1 \atop #3}\ {#2 \atop #4}\right)}

\title[Distinguished boundaries of $\mathbb{G}_2, \mathbb{E}$ and $\mathbb{P}$]{A new characterization for the distinguished boundaries of a few domains related to $\mu$-synthesis}
\author[Tomar]{NITIN TOMAR}

\address[Nitin Tomar]{Mathematics Department, Indian Institute of Technology Bombay, Powai, Mumbai-400076, India.} \email{tnitin@math.iitb.ac.in}		



\thanks{The author is supported by the Prime Minister Research Fellowship (PMRF ID 1300140), Government of India.}	
\begin{document}

	\maketitle
	
\begin{abstract}
In this note, we provide an alternative way of describing the  distinguished boundaries of symmetrized bidisc, tetrablock and pentablock. We show that each of these distinguished boundaries can be characterized as the union of orbits of two elements under their automorphism groups.
\end{abstract}

	\section{Introduction}
	
	\noindent For a bounded domain $\Omega \subset \C^n$, let $A(\Omega)$ be the algebra of continuous functions on its closure $\overline{\Omega}$ that are holomorphic in $\Omega$. A \textit{boundary} for $\Omega$ is a subset on which every function in $A(\Omega)$ attains its maximum modulus. The \textit{distinguished boundary} of $\Omega$ denoted by $b\overline{\Omega}$ is the smallest closed boundary of $\Omega$. If $\Omega$ is the unit disc $\D$ in the complex plane $\C$, then it follows from the maximum principle that $b\overline{\D}=\mathbb{T}$, the unit circle. An \textit{automorphism} of $\Omega$ is a bijective holomorphic self-map on $\Omega$ with holomorphic inverse. We denote the group of automorphisms of $\Omega$ by $Aut(\Omega)$. It is well-known that any automorpshim of $\D$ is of the form $\w B_\alpha$ for some $\omega \in \T$, where
	\[
	B_\alpha(z)= \frac{z-\alpha}{\overline{\alpha}z-1} \quad (z \in \D)
	\]
	for some $ \alpha \in \D$. For   any $\xi \in \mathbb{T}$, the map $v(z)=\xi z$ is in $Aut(\D)$ such that $v(1)=\xi$. Conversely, $v(1) \in \mathbb{T}$ for any $v \in Aut(\D)$. Thus, $\T=\{v(1) : v \in Aut(\D)\}$ and so $\T$ is the orbit of $z=1$ under the group $Aut(\D)$. Motivated by this, we describe the distinguished boundary of several other domains via their automorphisms. The domains that we consider here are the symmetrized polydisc $\Gg$ (see \cite{AglerII, Edigarian}), the tetrablock $\E$ (see \cite{Abouhajar}) and the pentablock $\Pe$ (see \cite{AglerIV}), where
\begin{equation*}
\begin{split}
	&\Gg:=\left\{\left(\underset{1 \leq i \leq n}{\sum}z_i, \underset{1 \leq i<j \leq n}{\sum}z_iz_j, \dotsc, \underset{1 \leq i \leq n}{\prod}z_i\right) \in \C^n : z_1, \dotsc, z_n \in \D\right\}  \quad (n \geq 2), \\	
	& \mathbb{E}:=\{(a_{11}, a_{22}, det(A)) \in \C^3 : A=[a_{ij}]_{2 \times 2}, \ \|A\|<1\},\\
	& 
\mathbb{P}:=\{(a_{21}, tr(A), det(A)) \in \C^3 : A=[a_{ij}]_{2 \times 2}, \ \|A\|<1\}. 
\end{split}
\end{equation*}

These domains arise naturally while considering special cases of the $\mu$-synthesis problem in control engineering, e.g. see \cite{Doyle}. The \textit{structured singular value} $\mu_E$ relative to a linear subspace $E$ of $M_n(\C)$, the space of $n\times n$ matrices, is given by 
	\[
		\mu_E(A):=\big(\inf\{\|X\| \ : \ X \in E, \ \text{det}(I-AX)=0 \}\big)^{-1}, \qquad A \in M_n(\mathbb{C}).
		\]
If $E=M_n(\C)$, then $\mu_E(A)=\|A\|$ and if $E$ is the space of all scalar multiples of identity, then $\mu_E(A)$ is the spectral radius $r(A)$. For any $A \in M_n(\C), r(A)<1$ if and only if the symmetrization of its eigenvalues belongs to $\Gg$ (see \cite{Edigarian}, Section 5) . If $E$ is the space of $2 \times 2$ diagonal matrices, then for any $A=[a_{ij}] \in M_2(\C), \mu_E(A)<1$ if and only if $(a_{11}, a_{22}, det(A)) \in \E$ (see \cite{Abouhajar}, Section 9).  Another natural choice of $E$ is the linear span of the identity matrix and $\begin{bmatrix}
0 & 1 \\
0 & 0
\end{bmatrix}$. In this case, for any $A=[a_{ij}] \in M_2(\C), \mu_E(A)<1$ if and only if $(a_{21}, tr(A), det(A)) \in \Pe$ (see \cite{AglerIV}, Section 3). These domains achieved considerable attentions in the recent years due to their geometric, function theoretic and operator theoretic aspects. Let the closure of $\mathbb{G}_n, \mathbb{E}$ and $\mathbb{P}$ be denoted by $\Gamma_n, \Ebar$ and $\Pbar$ respectively. The distinguished boundaries of $\G, \E$ and $\Pe$ were determined in \cite{AglerII}, \cite{Abouhajar} and \cite{AglerIV} to be 
\begin{equation*}
\begin{split}
& b\Gamma_2=\{(s, p) \in \C^2 : |s| \leq 2, s=\overline{s}p, |p|=1\} 
=\{(z_1+z_2, z_1z_2) \in \C^2 : |z_1|=|z_2|=1\}, \\
& b\Ebar=\{(x_1, x_2, x_3) \in \C^3 : |x_1| \leq 1, x_1=\overline{x}_2x_3, |x_3|=1\} \quad 
\text{and} \\
&b\Pbar=\{(a, s, p) \in \C^3 : |a|^2+|s|^2\slash 4=1, (s, p) \in b\Gamma_2\},\\
\end{split}
\end{equation*}
respectively. This short note provides a new description of the distinguished boundaries of $\G, \E$ and $\Pe$ as the orbits of certain elements under their respective automorphism groups.
\section{Main results}

\noindent To begin with, we briefly recollect the description of the automorphism groups of $\G, \E$ and $\Pe$ from the literature. The automorphisms of $\G$ are determined by the elements in $Aut(\D)$. Note that for $v \in Aut(\D)$, the map 
\[
\tau_v(z_1+z_2, z_1z_2)=(v(z_1)+v(z_2), v(z_1)v(z_2)) \quad (z_1, z_2 \in \Dbar)
\] 
is in $Aut(\G)$. Indeed, every automorphism of $\G$ is of this kind as the following theorem explains.
\begin{thm}[Jarnicki and Pflug, \cite{Jarnicki}]
$Aut(\G)=\{\tau_v : v \in Aut(\D)\}$.
\end{thm}
The group $Aut(\mathbb{E})$ is quite large. One trivial autmorphism of $\E$ is the flip map $F: \E \to \E$ given by $F(x_1, x_2, x_3)=(x_2, x_3, x_1)$. To define others, we  fix $v=-\xi_1B_{z_1}$ and $\chi=-\xi_2B_{-\overline{z}_2}$ for $\xi_1, \xi_2 \in \T$ and $z_1, z_2 \in \D$. The map
		\[
		T_{v, \chi}(x_1, x_2, x_3)=\bigg(T_1(x_1, x_2, x_3), \ T_2(x_1, x_2, x_3), \ T_3(x_1, x_2, x_3)\bigg) \quad ((x_1, x_2, x_3) \in \overline{\E} ),
		\]	
		where
		\begin{equation}\label{eqn2.1}
			\begin{split}
				& T_1(x_1, x_2, x_3)=\xi_1\frac{(x_1-z_1)+\xi_2\overline{z_2}(z_1x_2-x_3)}{(1-\overline{z}_1x_1)-\xi_2\overline{z}_2(x_2-\overline{z}_1x_3)},
				\quad  T_2(x_1, x_2, x_3)=\frac{z_2(\overline{z}_1x_1-1)+\xi_2(x_2-\overline{z}_1x_3)}{(1-\overline{z}_1x_1)-\xi_2\overline{z}_2(x_2-\overline{z}_1x_3)},\\
				\\
				& T_3(x_1, x_2, x_3)=\xi_1\frac{z_2(z_1-x_1)-\xi_2(z_1x_2-x_3)}{(1-\overline{z_1}x_1)-\xi_2\overline{z}_2(x_2-\overline{z}_1x_3)}		
			\end{split}
		\end{equation}
is in $Aut(\E)$. The automorphisms of $\E$ were first described in \cite{Abouhajar}. An interested reader may also refer to Section 6 of \cite{KosinskiII}. Finally, $Aut(\E)$ was completely described in \cite{Young} in the following way.
\begin{thm}[\cite{Young}, Theorem 4.1]
$Aut(\E)=\{T_{v, \chi}, T_{v, \chi}\circ F : v, \chi \in Aut(\D)  \}$.
\end{thm}  	
Now, we move to consider the automorphisms of the pentablock. For $\omega \in \T$ and $v=\eta B_\alpha \in Aut(\D)$, let 
\[
f_{\omega v}(a, s, p)=\bigg(\frac{\omega \eta (1-|\alpha|^2)a}{1-\overline{\alpha}s+\overline{\alpha}^2p}, \tau_{v}(s, p) \bigg) \quad ( (a, s, p) \in \Pbar).
\]
It was proved in \cite{AglerIV} that $\{f_{\omega v}: \omega \in \T, v \in Aut(\D)\}$ is a subgroup of $Aut(\Pe)$ . Indeed, we have the following:
\begin{thm}[\cite{Kosinski}, Theorem 15]
$Aut(\Pe)=\{f_{\omega v} :  \omega \in \T, v \in Aut(\D) \}$.
\end{thm}	
We now present the main results of this article. We first consider the symmetrized bidisc $\G$ as a special case.
\begin{thm}\label{thm1}
$b\Gamma_2=\{\tau_v(0,1), \tau_v(2,1) : v \in Aut(\D)\}$. Furthermore, 
\[
b\Gamma_2 \cap V(\G)=\{ \tau_v(2,1) : v \in Aut(\D)\} \quad \text{and} \quad b\Gamma_2 \setminus V(\G)=\{\tau_v(0,1) : v \in Aut(\D)\}
\]
where $V(\G)=\{(s, p) \in \Gamma_2 : s^2=4p\}$ is the royal variety in $\G$.
\end{thm}
\begin{proof}
The royal variety $V(\G)=\{(s, p) \in \Gamma_2 : s^2=4p\}$ was defined in \cite{AglerIII}. Evidently, $(0, 1) \in b\Gamma_2 \setminus V(\G)$ and $(2,1) \in b\Gamma_2 \cap V(\G)$. It was proved in \cite{AglerV} that  any automorphism of $\G$ leaves $V(\G)$ invariant. Thus, we have that  $\tau_v(0,1) \in b\Gamma_2 \setminus V(\G)$ and $\tau_v(2,1) \in b\Gamma_2 \cap V(\G)$. Let $(s, p) \in b\Gamma_2 \setminus V(\G)$. Then there exist $z_1, z_2 \in \T$ such that $z_1 \ne z_2$ and $(s, p)=(z_1+z_2, z_1z_2)$. Choose $v \in Aut(\D)$ such that $v(i)=z_1$ and $v(-i)=z_2$. So, $\tau_{v}(0, 1)=(v(i)+v(-i), v(i)v(-i))=(s,p)$. Consequently, we have
\[
 b\Gamma_2 \setminus V(\G)=\{\tau_v(0,1) : v \in Aut(\D)\}.
\]
Let $(s, p) \in b\Gamma_2 \cap V(\G)$. Then there there exists $z \in \T$ such that $(s, p)=(2z, z^2)$. Take $v \in Aut(\D)$ so that $v(1)=z$. Then $\tau_v(2,1)=(2v(1), v(1)^2)=(s, p)$. Thus, 
\[
b\Gamma_2 \cap V(\G)=\{ \tau_v(2,1) : v \in Aut(\D)\}.
\]
The proof is complete.
\end{proof}
It was shown in \cite{Edigarian} that the distinguished boundary $b\Gamma_n$ of $\Gg$ consists of the point $\pi_n(\T^n)$, where $\pi_n$ is the symmetrization map given by
\[
\pi_n(z_1, \dotsc, z_n)=\left(\underset{1 \leq i \leq n}{\sum}z_i, \underset{1 \leq i<j \leq n}{\sum}z_iz_j, \dotsc, \underset{1 \leq i \leq n}{\prod}z_i \right).
\]
Any proper holomorphic map $\tau_B: \Gg \to \Gg$ is of the form
\[
\tau_B(\pi_n(z_1, \dotsc, z_n))=\pi_n(B(z_1), \dotsc, B(z_n)) \quad (z_1, \dotsc, z_n \in \D)
\]
where $B$ is a finite Blaschke product. For a detailed proof of this result, one may refer Theorem 1 in \cite{Edigarian}. For $n \in \mathbb N$, let $\mu_1, \dotsc, \mu_n$ be distinct $n$-th roots of unity. Then each $B(\mu_j) \in \T$ for a fnite Blaschke product $B$. Thus, $\pi_n(B(\mu_1), \dotsc, B(\mu_n)) \in b\Gamma_n$. Moreover, Cantor and Phelps \cite{Cantor} showed that for $z_1, \dotsc, z_n \in \T$, there exists a finite Blaschke product $B$ such that $B(\mu_j)=z_j$ for $1 \leq j \leq n$. Then $\tau_B(\pi_n(\mu_1,\dotsc, \mu_n))=\pi_n(B(\mu_1), \dotsc, B(\mu_n))=\pi_n(z_1, \dotsc, z_n)$. Thus, the following result holds.
\begin{thm}: Let $n \geq 2$ and let $\mu_1, \dotsc, \mu_n$ be distinct $n$-th roots of unity. Then
\[
b\Gamma_n=\{\tau_B(\pi_n(\mu_1, \dotsc, \mu_n)) : \mbox{$B$ is a finite Blaschke product}\}.
\]
\end{thm} 

\begin{thm}\label{thm2}
	$b\Ebar=\{T_{v, \chi}(0, 0, 1), T_{v, \chi}(1,1,1) \ : \ v, \chi \in Aut(\D)\}$. Moreover, 
	\[
	b\Ebar \cap \Delta(\overline{\E})=\{T_{v, \chi}(1,1,1) \ : \ v, \chi \in Aut(\D) \}
	\quad
	\text{and} 
	\quad 
	b\Ebar \setminus \Delta(\overline{\E})=\{T_{v, \chi}(0,0,1) \ : \ v, \chi \in Aut(\D) \},
	\]	
	where $\Delta(\overline{\E})=\{(x_1, x_2, x_3) \in \Ebar : x_1x_2=x_3\}$, the collection of triangular points in $\Ebar$. 
\end{thm}
\begin{proof} 
It is easy to see that $(0, 0, 1) \in b\Ebar \setminus \Delta(\overline{\E})$.  By (\ref{eqn2.1}), we have 
	
		\begin{equation}\label{eqn6.2}
		\begin{split}
			T_{v, \chi}(0, 0, 1)=\bigg( -\xi_1\frac{z_1+\xi_2\overline{z}_2}{1+\xi_2\overline{z}_1\overline{z}_2}, \ \
			\frac{-(\overline{z}_1\xi_2+z_2)}{1+\xi_2\overline{z}_1\overline{z}_2}, \ \
			\xi_1\overline{\xi}_2\frac{1+z_1z_2\overline{\xi}_2}{1+\xi_2\overline{z}_1\overline{z}_2}
			\bigg),		
		\end{split}
	\end{equation}
	where 	$v=-\xi_1B_{z_1}, \chi=-\xi_2B_{-\overline{z}_2}$
	for some $\xi_1, \xi_2 \in \T$ and $z_1, z_2 \in \D$.
	It is not difficult to see that the third component of $T_{v, \chi}(0,0,1)$ is of unit modulus for any $v, \chi \in Aut(\D)$. Now, we show that $T_{v,\chi}(0, 0, 1)$ is not a triangular point. It was proved in \cite{Abouhajar} that if $(x_1, x_2 ,x_3) \in b\Ebar$, then either $x_1x_2 \ne x_3, |x_1|<1$ and $|x_3|=1$ or $x_1x_2=x_3$ and $x_1, x_2, x_3 \in \T$. So, if $T_{v,\chi}(0, 0, 1)$ is a triangular point, then  the first component of $T_{v, \chi}(0,0,1)$ must have unit modulus. We have that
	\begin{equation*}
		\begin{split}
			\bigg|\frac{z_1+\xi_2\overline{z}_2}{1+\xi_2\overline{z}_1\overline{z}_2}\bigg|=1 
			 \iff 	|z_1+\xi_2\overline{z}_2|^2=|1+\xi_2\overline{z}_1\overline{z}_2|^2
			 \iff (1-|z_1|^2)(1-|z_2|^2)=1,
		\end{split}
	\end{equation*}
	which is not possible as $z_1, z_2 \in \D$. Thus, $T_{v, \chi}(0, 0,1) \in b\Ebar \setminus \Delta(\overline{\E})$ for any $v, \chi \in Aut(\D)$. Conversely, let $x=(x_1, x_2, x_3) \in b\Ebar \setminus \Delta(\overline{\E})$. Choose $z_1=-x_1\overline{\xi}_1, z_2=0,\xi_1=x_3$ and $\xi_2=1$. Since $(x_1, x_2, x_3) \in b\Ebar \setminus \Delta(\Ebar)$, we have by the previous discussion that $|x_1|<1$ and so $z_1 \in \D$. Define $v=-\xi_1B_{z_1}$ and $\chi=-\xi_2B_{-\overline{z}_2}$, it follows from (\ref{eqn6.2}) that $
	T_{v, \chi}(0, 0, 1)=(x_1,  \overline{x}_1x_3, x_3)=(x_1, x_2, x_3)$.
	Thus,
	\[
	b\Ebar \setminus \Delta(\overline{\E})=\bigg\{T_{v, \chi}(0,0,1) \ : \ v, \chi \in Aut(\D) \bigg\}.
	\]
	Evidently, $(1, 1, 1) \in b\Ebar \cap \Delta(\overline{\E})$. Take $v=-\xi_1B_{z_1}, \chi=-\xi_2B_{-\overline{z}_2}$ 
	for some $\xi_1, \xi_2 \in \T$ and $z_1, z_2 \in \D$. By (\ref{eqn2.1}), we have
	\begin{equation}\label{eqn6.3}
		\begin{split}
			T_{v, \chi}(1, 1, 1)=\bigg( \xi_1\frac{1-z_1}{1-\overline{z}_1}, \ \
			\frac{\xi_2-z_2}{1-\xi_2\overline{z}_2}, \ \
			\xi_1\frac{(1-z_1)(\xi_2-z_2)}{(1-\overline{z}_1)(1-\xi_2\overline{z}_2)}
			\bigg).		
		\end{split}
	\end{equation}
	It is easy to see that $T_{v, \chi}(1,1,1) \in b\Ebar \cap \Delta(\Ebar)$ for any $v, \chi \in Aut(\D)$. To see the converse, let $x=(x_1, x_2, x_3) \in b\Ebar \cap \Delta(\overline{\E})$. Then $x_1, x_2$ and $x_3$ are in $\T$. Choose $
	z_1=0, z_2=0, \xi_1=x_1$ and $\xi_2=x_2$. For $v=-\xi_1B_{z_1}$ and $\chi=-\xi_2B_{-\overline{z}_2}$, we have by (\ref{eqn6.3}) that $T_{v, \chi}(1, 1, 1)=(x_1,  x_2, x_1x_2)=(x_1, x_2, x_3)$. Thus, 
	\[
	b\Ebar \cap \Delta(\overline{\E})=\bigg\{T_{v, \chi}(1,1,1) \ : \ v, \chi \in Aut(\D) \bigg\}.
	\]
The proof is complete.	
\end{proof}
Now, we descirbe the distinguished boundary of the pentablock $\Pe$. Recall that for $h \in A(\Pe)$ and $x \in \Pbar$ such that $h(x)=1$ and $|h(y)|<1$ for all $y \in \Pbar \setminus \{x\}$, we call $x$ a \textit{peak point} of $\Pbar$ and $h$ a \textit{peaking function} for $x$.
\begin{thm}\label{thm3}
$b\Pbar=\{f_{\omega v}(1, 0,1), f_{\omega v}(0, 2, 1) : \omega \in \T, v \in Aut(\D) \}$. Moreover, 
\[
b\Pbar \cap V(\Pe)=\{f_{\omega v}(0, 2, 1) : \omega \in \T, v \in Aut(\D) \}  \  \ \ \text{and} \ \ \ \ b\Pbar \setminus V(\Pe)=\{f_{\omega v}(1, 0,1) : \omega \in \T, v \in Aut(\D) \}
\]
where $V(\Pe)=\{(a, s, p) \in \Pbar : s^2=4p\}$.
\end{thm}

\begin{proof}
It follows from Theorem 7.1 in \cite{Abouhajar} that every point of $b\Pbar$ is a peak point of $\Pbar$.  By definition, any peak point of $\Pbar$ must be in $b\Pbar$. Moreover, if $f$ is an automorphism of $\Pe$ and $x$ is a peak point of $\Pbar$ with a peaking function $h$, then $f(x)$ is also a peak point of $\Pbar$ with peaking function $h\circ f^{-1}$. Thus, $Aut(\Pe)$ preserves peak points of $\Pbar$ and consequently, $b\Pbar$. Hence, $f_{\omega v}(1, 0,1), f_{\omega v}(0, 2, 1)$ are in $b\Pbar$ for any $f_{\omega v} \in Aut(\Pe)$. Let $(a_0, s_0, p_0) \in b\Pbar$. If $s_0^2=4p_0$, then  $a_0=0$ since $|p_0|=1$ and $|a_0|^2+|s_0|^2\slash 4=1$. It follows from Theorem \ref{thm1} that there is an automorphism $v$ of $\D$ so that $\tau_{v}$ is in $Aut(\G)$ and $\tau_v(2, 1)=(s_0, p_0)$. Then for any $\omega \in \T$, the automorphism map $f_{\omega v}$ of $\Pe$ maps $(0, 2, 1)$ to $(0, \tau_v(2, 1))=(a_0, s_0, p_0)$. Thus, $
b\Pbar \cap V(\Pe)=\{f_{\omega v}(0, 2, 1) : \omega \in \T, v \in Aut(\D) \}$.

\medskip 

Let $s_0^2 \ne 4p_0$. Then again by Theorem \ref{thm1}, one can choose an automorphism $\tau_v$ of $\G$ for some $v=\eta B_\alpha \in Aut(\D)$ such that $\tau_v(0, 1)=(s_0, p_0)$. The map given by 
\[
f_{\omega v}(a, s, p)=\bigg(\frac{\omega \eta (1-|\alpha|^2)a}{1-\overline{\alpha}s+\overline{\alpha}^2p}, \tau_{v}(s, p) \bigg) \quad ( (a, s, p) \in \Pbar)
\]	
is in $Aut(\Pe)$, where $\omega \in \T$ is to be determined later. Evidently, $f_{\omega v}(1, 0, 1)=(\omega \widetilde{a}, s_0, p_0)$ where $\widetilde{a}=\eta (1-|\alpha|^2)(1+\overline{\alpha}^2)^{-1}$. From the discussion above, we have $(\omega \widetilde{a}, s_0, p_0) \in b\Pbar$ for $\omega \in \T$, in particular,$(\widetilde{a}, s_0, p_0) \in b \Pbar$. Therefore,  $|\widetilde{a}|^2+|s_0|^2\slash 4=1$ and so $|\widetilde{a}|=|a_0|$. Choose $\omega \in \T$ so that $a_0=\omega \widetilde{a}$. Thus, $f_{\omega v}(1, 0, 1)=(a_0, s_0, p_0)$. Thus, $b\Pbar \setminus V(\Pe)=\{f_{\omega v}(1, 0, 1) : \omega \in \T, v \in Aut(\D) \}$. The proof is now complete.
\end{proof}


\begin{thebibliography}{9}
		
		\bibitem{Abouhajar}
		A. A. Abouhajar, M. C. White and N. J. Young, \textit{A Schwarz lemma for a domain related to $\mu$-synthesis}, J. Geom. Anal., 17 (2007), 717 -- 750.\\
		
		\bibitem{AglerIV}	
		J. Agler, Z. A. Lykova and N. J. Young, \textit{The complex geometry of a domain related to $\mu$-synthesis}, J. Math. Anal. Appl., 422 (2015), 508 -- 543.\\		
		

		\bibitem{AglerII}	
		J. Agler and N. J. Young, \textit{A commutant lifting theorem for a domain in $\C^2$ and spectral interpolation}, J. Funct. Anal., 161 (1999), 452 -- 477.\\
		
		\bibitem{AglerIII}
		J. Agler and N. J. Young, \textit{The hyperbolic geometry of the symmetrized bidisc}, J. Geom. Anal., 14 (2004), 375 -- 403.\\
		

		
\bibitem{AglerV}
J. Agler and N. J. Young, \textit{The magic functions and automorphisms of a domain}, Complex Anal. Oper. Theory, 2 (2008), 383 -- 404.\\		
		
\bibitem{Cantor}
D. G. Cantor and R.  R. Phelps, \textit{An elementary interpolation theorem}, Proc. Amer. Math. Soc., 16 (1965), 523 -- 525.\\
		
				\bibitem{Doyle}
		J. Doyle, \textit{Analysis of feedback systems with structured uncertainties}, IEE Proc. Control Theory Appl., 129 (1982), 242 -- 250.\\

\bibitem{Edigarian}
A. Edigarian and W. Zwonek, \textit{Geometry of the symmetrized polydisc}, Arch. Math., 84 (2005), 364 -- 374.\\
		
\bibitem{Jarnicki}
		M. Jarnicki and P. Pflug, \textit{On automorphisms of the symmetrized bidisc}, Arch. Math. (Basel), 83 (2004), 264 -- 266.\\

\bibitem{Kosinski}
L. Kosi\'{n}ski, \textit{The group of automorphisms of the pentablock}, Complex Anal. Oper. Theory, 6 (2015), 1349 -- 1359.\\ 


\bibitem{KosinskiII}
L. Kosi\'{n}ski and W. Zwonek, \textit{Nevanlinna-Pick problem and uniqueness of left inverses in convex domains, symmetrized bidisc and tetrablock}, J. Geom. Anal., 26 (2016), 1863 -- 1890.\\ 

		\bibitem{Young}
		N. J. Young, \textit{The automorphism group of the tetrablock}, J. Lond. Math. Soc., 77 (2008), 757 -- 770.\\
		
	\end{thebibliography}
\end{document}